\documentclass{amsart}

\usepackage{graphicx,color}
\usepackage{here}

\newtheorem{theorem}[subsection]{Theorem}
\newtheorem{lemma}[subsection]{Lemma}

\theoremstyle{definition}
\newtheorem{definition}[subsection]{Definition}

\theoremstyle{remark}
\newtheorem{remark}[subsection]{Remark}

\title{Minimal coloring number for $\mathbb{Z}$-colorable links II}

\author{Eri Matsudo}
\address{Graduate School of Integrated Basic Sciences, Nihon University,
3-25-40 Sakurajosui, Setagaya-ku, Tokyo 156-8550, Japan}
\email{cher16001@g.nihon-u.ac.jp}


\date{\today}

\begin{document}

\maketitle

\begin{abstract}
The minimal coloring number of a $\mathbb{Z}$-colorable link 
is the minimal number of colors for non-trivial $\mathbb{Z}$-colorings on diagrams of the link.
In this paper, we show that the minimal coloring number of any non-splittable $\mathbb{Z}$-colorable links is four. 
As an example, we consider the link obtained 
by replacing each component of the given link 
with several parallel strands, 
which we call a parallel of a link.
We show that an even parallel of a link is $\mathbb{Z}$-colorable except for the case of 2 parallels with non-zero linking number. 
We then give a simple way 
to obtain a diagram which attains the minimal coloring number 
for such even parallels of links.
\end{abstract}

\section{Introduction}

In \cite{Fox}, 
Fox introduced one of the most well-known invariants for knots and links, 
which now it is called \textit{the Fox $n$-coloring}, or simply {\it $n$-coloring} for $n\ge 2$. 

On the other hand, it is known that the links with $0$ determinants cannot admit  
Fox $n$-coloring for any $n \ge 2$. 
For such links, the $\mathbb{Z}$-coloring can be defined 
as a generalization of  the Fox $n$-coloring. 
In \cite{HK}, Haraly-Kauffman defined the minimal coloring number for Fox $n$-coloring. 
We define the minimal coloring number for $\mathbb{Z}$-coloring 
as a generalization of the minimal coloring number for Fox $n$-coloring.
See Section 2.

The minimal coloring number of any splittable $\mathbb{Z}$-colorable link is shown to be $2$.
In \cite{IM}, 
for  a non-splittable $\mathbb{Z}$-colorable link $L$ that has a diagram with a ``simple'' $\mathbb{Z}$-coloring, 
we proved that the minimal coloring number of $L$ is $4$. 
In this paper, we show that any non-splittable $\mathbb{Z}$-colorable link has a diagram with a ``simple'' $\mathbb{Z}$-coloring, and its minimal coloring number is $4$. \\

\noindent
\textbf{Theorem \ref{main}.} {\it The minimal coloring number of any non-splittable $\mathbb{Z}$-colorable link is equal to $4$.}\\

This result is also proved 
by Meiqiao Zhang, Xian'an Jin and Qingying Deng almost independently in \cite{Zhang}. 
Previously Zhang gave us her manuscript for her Master thesis. 
There she showed that if a $\mathbb{Z}$-colorable link has a diagram with a $1$-diff crossing, the link has a diagram with only $0$-diff crossings and $1$-diff crossings. 
Our proof is based on her arguments.  

In the proof of Theorem \ref{main}, 
we give a procedure to obtain a diagram with a $\mathbb{Z}$-coloring of $4$ colors 
from any given diagram with a non-trivial $\mathbb{Z}$-coloring 
of a non-splittable $\mathbb{Z}$-colorable link.
However, from a given diagram of a $\mathbb{Z}$-colorable link, 
by using the procedure given in the our proof of Theorem \ref{main}, 
the obtained diagram and $\mathbb{Z}$-coloring might be very complicated.

In Section \ref{secthm1}, we give a ``simple'' diagrams and $\mathbb{Z}$-coloring of $4$ colors 
for some particular class of $\mathbb{Z}$-colorable link. 
In fact, we consider the link obtained 
by replacing each component of the given link 
with several parallel strands, 
which we call a parallel of a link.


\section{Preliminaries}

Let us begin with the definition of $\mathbb{Z}$-coloring of link.

\begin{definition}\label{def1}
Let $L$ be a link and $D$ a regular diagram of $L$. 
We consider a map $\gamma:\{$arcs of $D\}\rightarrow \mathbb{Z}$. 
If $\gamma$ satisfies the condition 
$2\gamma(a)= \gamma(b)+\gamma(c)$ 
at each crossing of $D$ 
with the over arc $a$ and the under arcs $b$ and $c$, 
then $\gamma$ is called 
a \textit{$\mathbb{Z}$-coloring} on $D$. 
A $\mathbb{Z}$-coloring which assigns the same integer to all the arcs of the diagram 
is called the \textit{trivial $\mathbb{Z}$-coloring}. 
A link is called \textit{$\mathbb{Z}$-colorable} if it has a diagram admitting a non-trivial
$\mathbb{Z}$-coloring. 
\end{definition}

Throughout this paper, we often call the integers of the image of a $\mathbb{Z}$-coloring {\it colors}. 

We define the minimal coloring number for $\mathbb{Z}$-coloring as follows.

\begin{definition}\label{def2}
Let us consider the number of the colors for 
a non-trivial $\mathbb{Z}$-coloring on a diagram 
of a $\mathbb{Z}$-colorable link $L$. 
We call the minimum of such number of colors 
for all non-trivial $\mathbb{Z}$-colorings on diagrams of $L$ 
the \textit{minimal coloring number} of $L$, 
and denote it by $mincol_\mathbb{Z}(L)$. 
\end{definition}

In \cite{IM}, we defined a simple $\mathbb{Z}$-coloring. 

\begin{definition}
Let $L$ be a non-trivial $\mathbb{Z}$-colorable link, 
and $\gamma$ a $\mathbb{Z}$-coloring on a diagram $D$ of $L$. 
Suppose that there exists a natural number $d$ such that, at all the crossings in $D$, 
the differences between the colors of the over arcs and the under arcs are $d$ or $0$.
Then we call $\gamma$ a \textit{simple} $\mathbb{Z}$-coloring.
\end{definition}

Moreover we have proved the following result in \cite{IM}.

\begin{theorem}{\cite[Theorem 4.2]{IM}}\label{simplethm}
Let $L$ be a non-splittable $\mathbb{Z}$-colorable link. 
If there exists a simple $\mathbb{Z}$-coloring on a diagram of $L$,
then $mincol_\mathbb{Z}(L)=4$.
\end{theorem}

\section{Main theorem}

In this section, we prove Theorem \ref{main}. 

\begin{theorem}\label{main}
The minimal coloring number of any non-splittable $\mathbb{Z}$-colorable link is equal to $4$.
\end{theorem}

In Zhang's thesis, 
she calls a crossing an {\it $n$-diff crossing} 
if $|b-a|$ and $|b-c|$ are equal to $n$, 
where the over arc is colored by $b$ and the under arcs are colored by $a$ 
and $c$ by a $\mathbb{Z}$-coloring $\gamma$ at the crossing. 
We also use this notion in our proof.

\begin{proof}[Proof of Theorem \ref{main}]
Let $L$ be a non-splittable $\mathbb{Z}$-colorable link.
If the link $L$ admits a simple $\mathbb{Z}$-coloring, 
by Theorem \ref{simplethm}, we see that $mincol_\mathbb{Z}(L)=4$.

Here let $D$ be a diagram of $L$ with non-simple $Z$-coloring. 
We define $d_m$ as the maximum of the set $\{0, d_1, d_2, \cdots, d_m\}$ 
such that $D$ has $d_i$-diff crossings for $i=1, 2, 3, \cdots$. 
We can find a path on $D$ from a $d_m$-diff crossing 
to a $d$-diff crossing passing only $0$-diff crossings with $0<d<d_m$. 
Such a path is one of the $4$ types [1], [2], [3] and [4] illustrated in Figure \ref{dm}. 

In all the figures in this proof, 
the crossing with $n$ inside a circle is an $n$-diff crossing.

\begin{figure}[H]
\begin{center}
\includegraphics[height=5cm,clip]{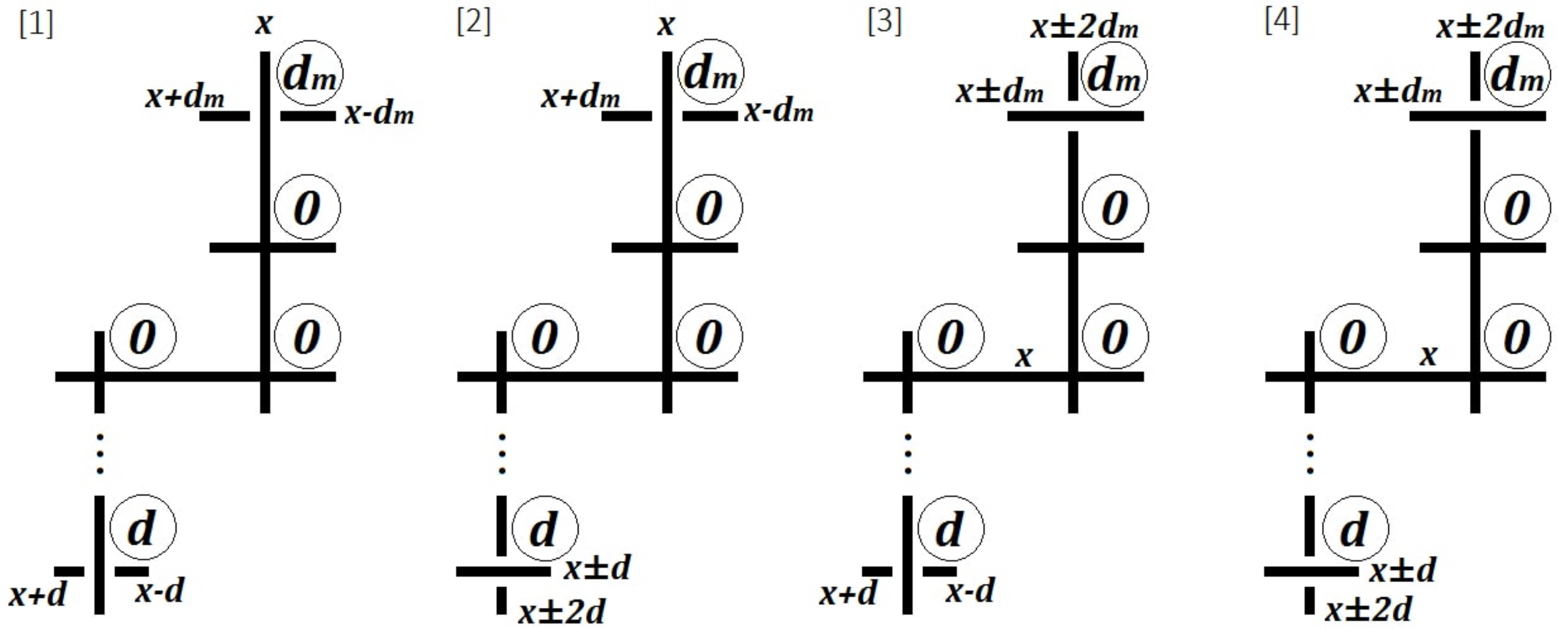}
\caption{}\label{dm}
\end{center}
\end{figure}

In the following, 
for a path from a $d_m$-diff crossing in Figure \ref{dm}, 
we will modify the diagram and the coloring to eliminate the $d_m$-diff crossing.

For a path of type [1], we modify the diagram and the coloring as shown in Figure \ref{fig1}.
\begin{figure}[H]
\begin{center}
\includegraphics[height=8cm,clip]{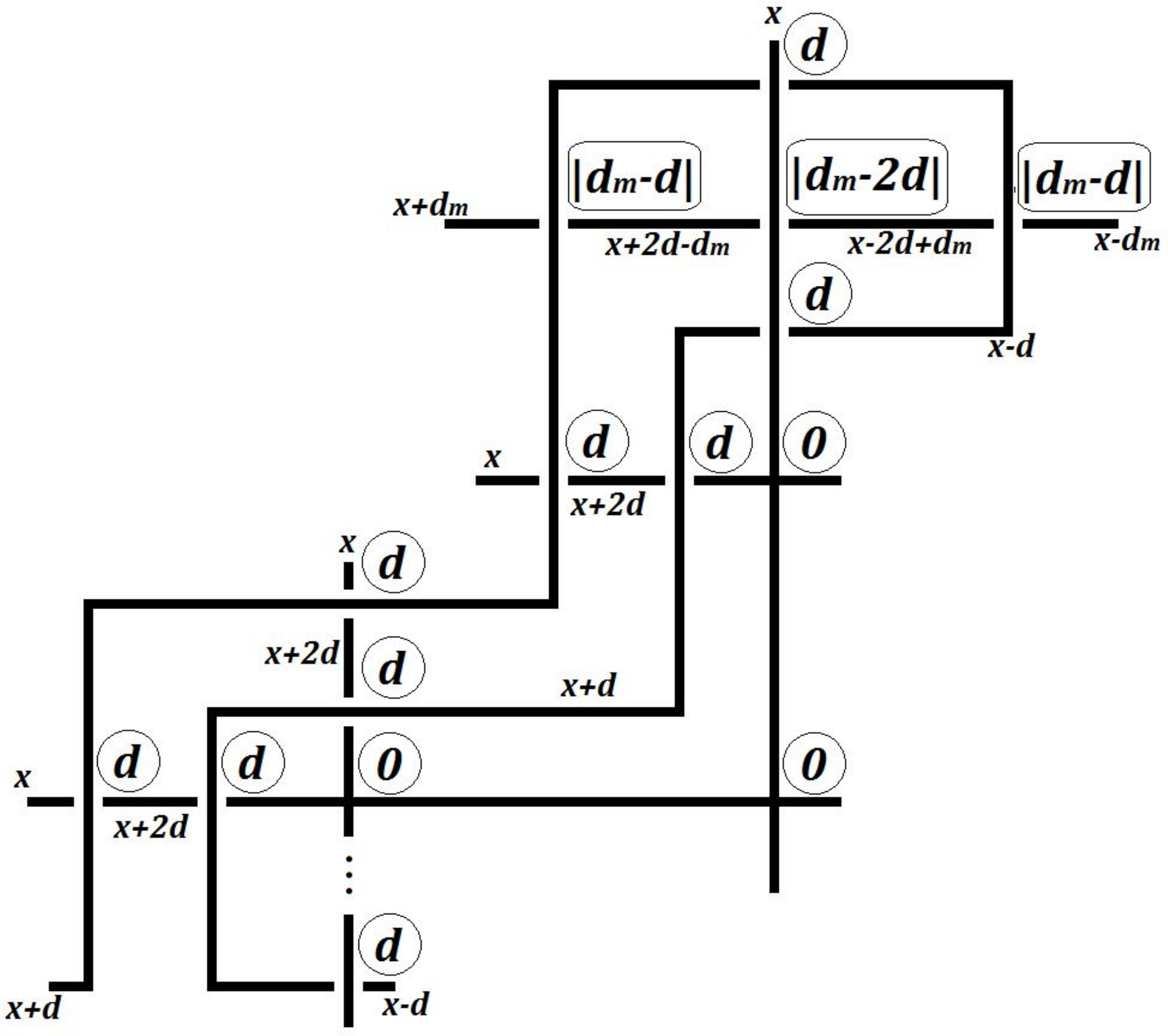}
\caption{}\label{fig1}
\end{center}
\end{figure}

For a path of type [2], we modify the diagram and the coloring as shown in Figure \ref{fig2-1} or \ref{fig2-2}.
\begin{figure}[H]
\begin{center}
\includegraphics[height=8cm,clip]{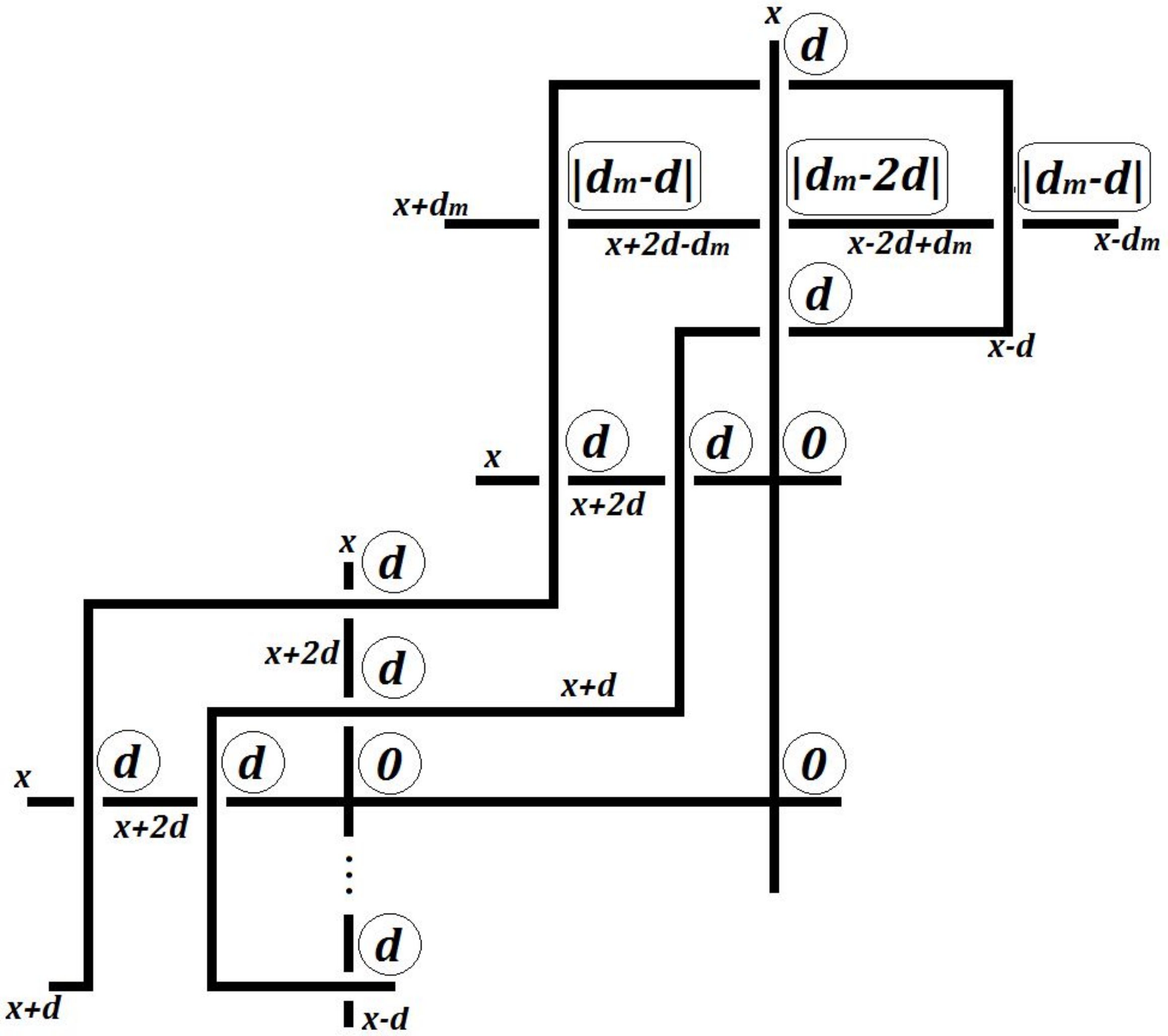}
\caption{}\label{fig2-1}
\end{center}
\end{figure}

\begin{figure}[H]
\begin{center}
\includegraphics[height=8cm,clip]{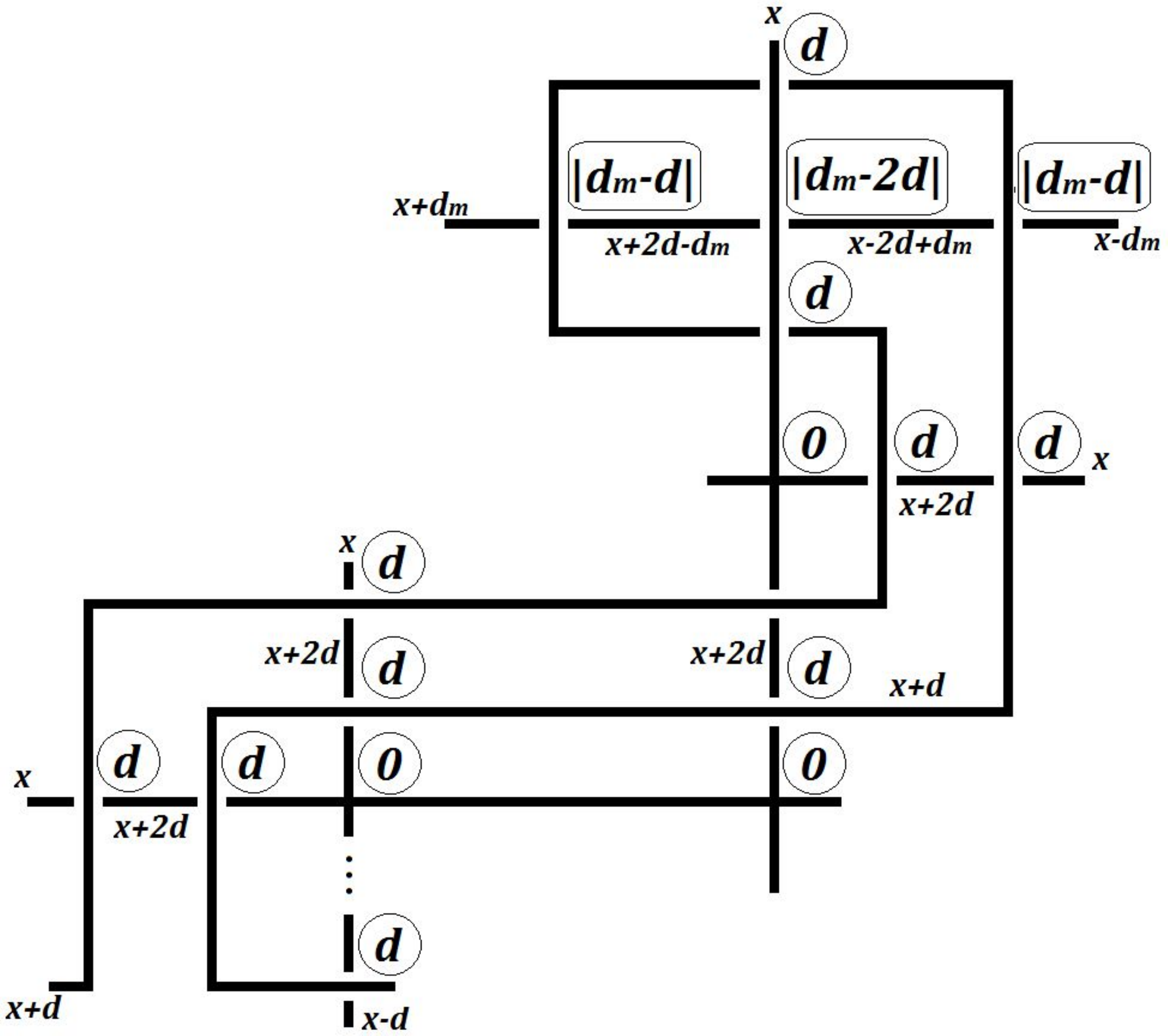}
\caption{}\label{fig2-2}
\end{center}
\end{figure}

For a path of type [3], modify the diagram and the coloring as shown in Figure \ref{fig3}.
\begin{figure}[H]
\begin{center}
\includegraphics[height=8cm,clip]{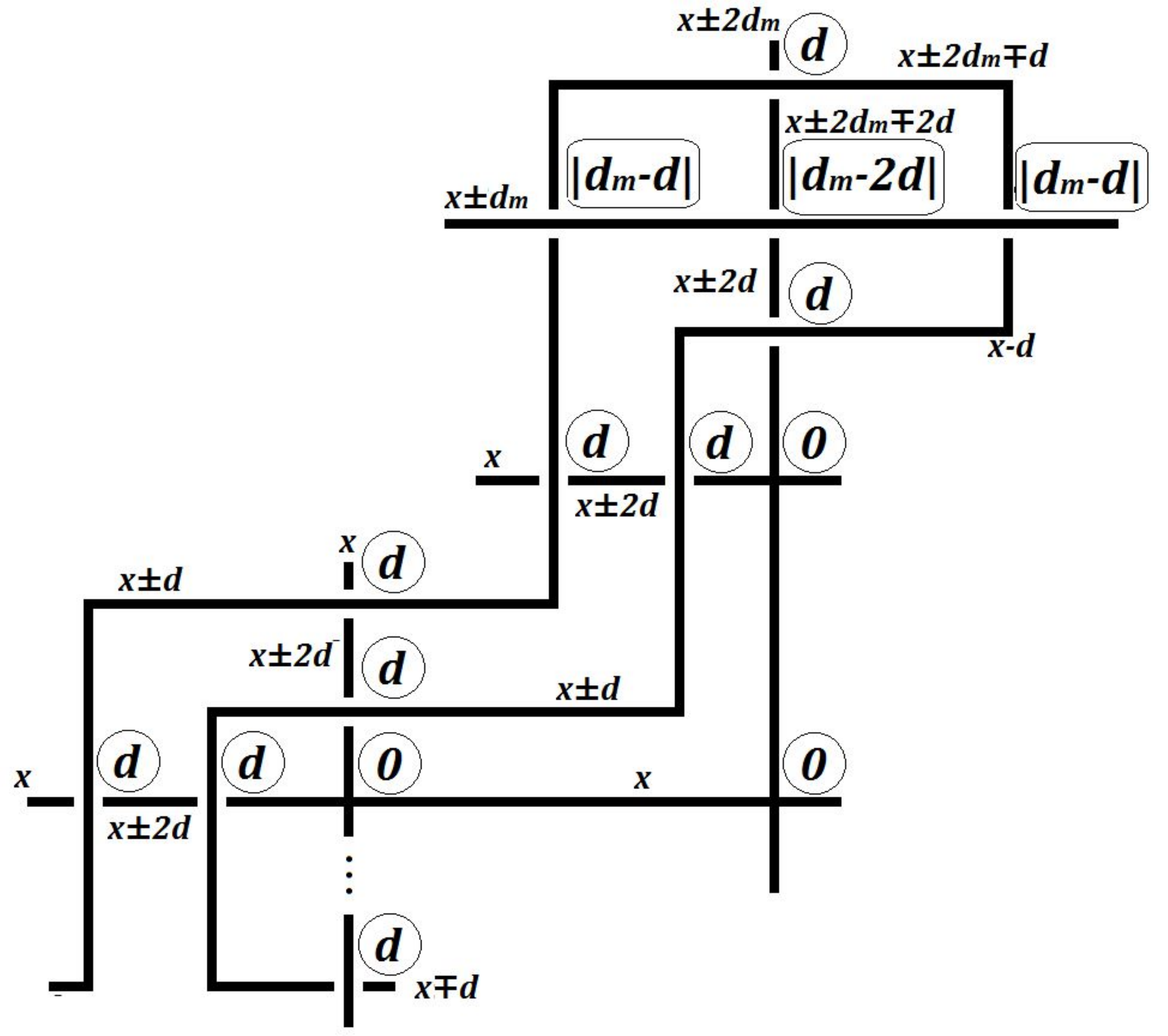}
\caption{}\label{fig3}
\end{center}
\end{figure}

For a path of type [4], modify the diagram and the coloring as shown in Figure \ref{fig4-1} or \ref{fig4-2}.
\begin{figure}[H]
\begin{center}
\includegraphics[height=8cm,clip]{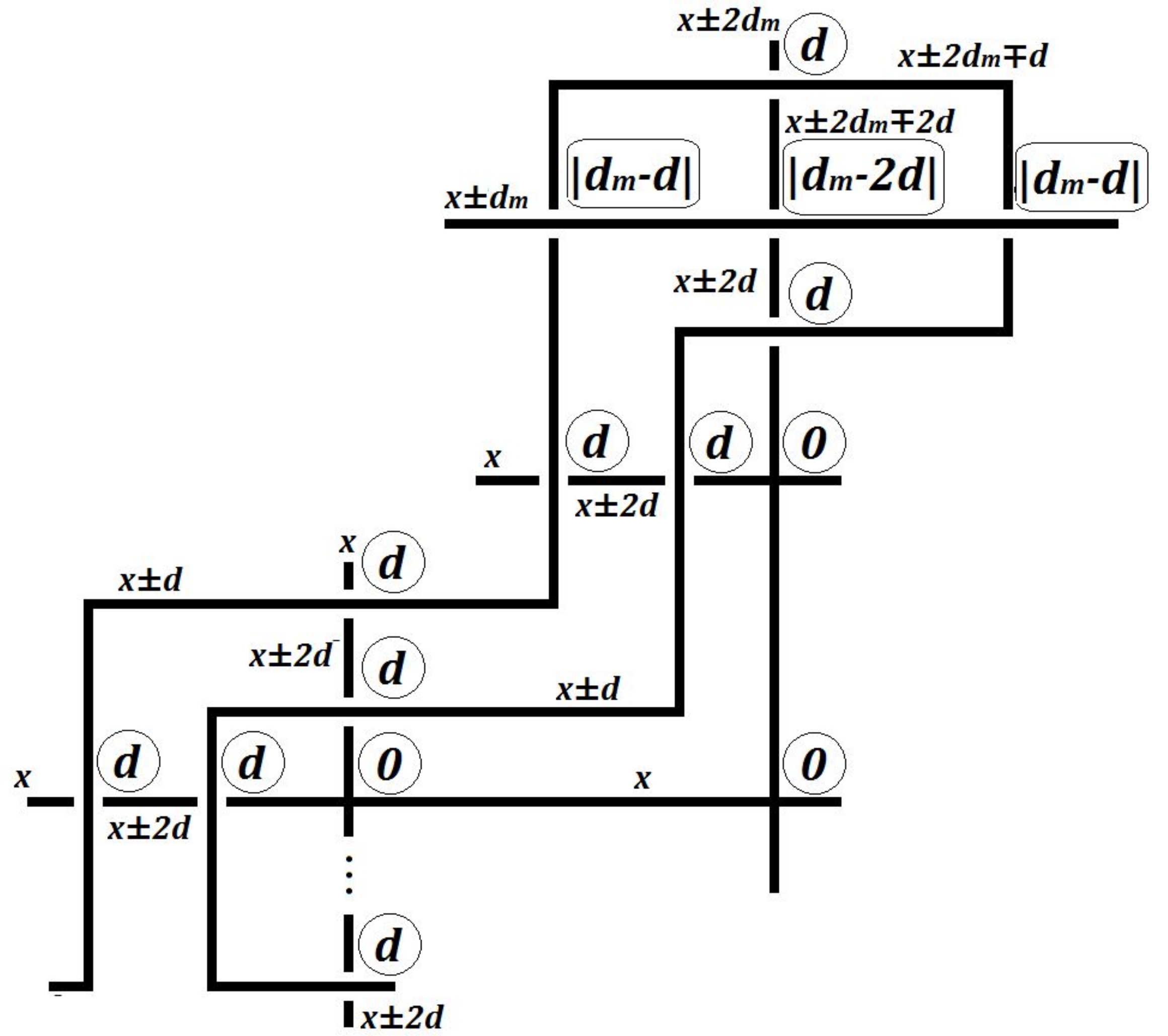}
\caption{}\label{fig4-1}
\end{center}
\end{figure}

\begin{figure}[H]
\begin{center}
\includegraphics[height=8cm,clip]{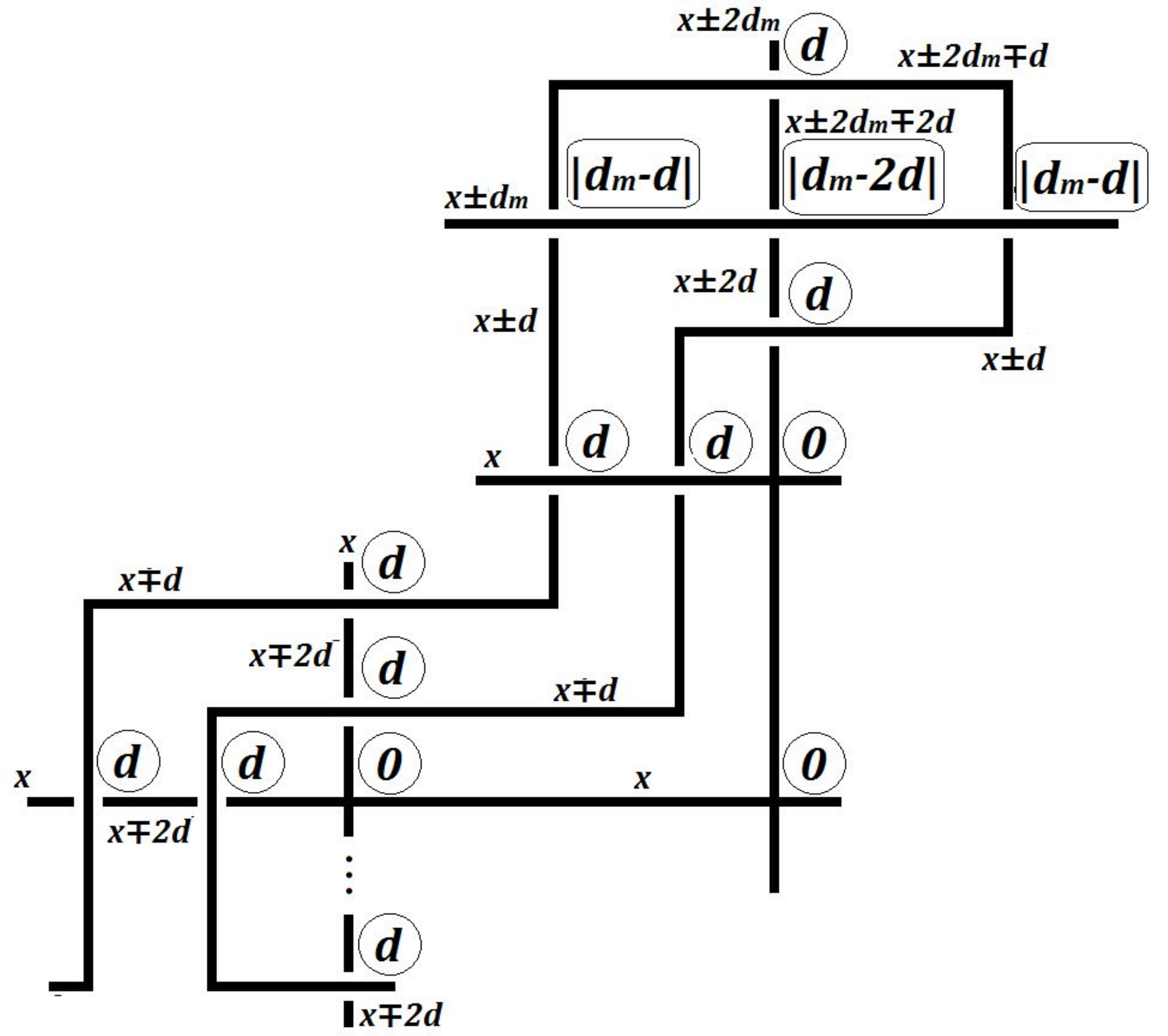}
\caption{}\label{fig4-2}
\end{center}
\end{figure}

Here the obtained diagram has $|d_m-d|$-diff crossings, $|d_m-2d|$-diff crossings 
and no $d_m$-diff crossings.
From $0<d<dm$, we see that $|d_m-d|$ and $|d_m-2d|$ are less than $d_m$. 
By the induction for $dm$, 
$L$ has a diagram with a $\mathbb{Z}$-coloring 
such that has only $0$-diff crossings and $\alpha$-diff crossings for $\alpha>0$. 
That is $L$ admits a simple $\mathbb{Z}$-coloring. 
By Theorem \ref{simplethm}, we conclude that $mincol_\mathbb{Z}(L)=4$.
\end{proof}

\begin{remark}
By Theorem \ref{main}, 
any non-splittable $\mathbb{Z}$-colorable link has a diagram
with a $\mathbb{Z}$-coloring of $4$ colors. 
However, from a given diagram of a $\mathbb{Z}$-colorable link, 
by using the procedure given in the our proof of Theorem \ref{main}, 
the obtained diagram and $\mathbb{Z}$-coloring might be very complicated.
\end{remark}

\section{Even parallels}\label{secthm1}

In this section, 
we give a simple way 
to obtain a diagram which attains the minimal coloring number 
for a particular family of $\mathbb{Z}$-colorable links.  
That is, we consider the link obtained 
by replacing each component of the given link 
with several parallel strands, 
which we call a parallel of a link, as follows. 

\begin{definition}
Let $L=K_1\cup\cdots\cup K_c$ be a link with $c$ components 
and $D$ a diagram of $L$. 
For a set $(n_1,\cdots ,n_c)$ of integers $n_i\geq 1$, 
we denote by $D^{(n_1,\cdots ,n_c)}$ the diagram obtained 
by taking $n_i$-parallel copies of the $i$-th component $K_i$ of $D$ 
on the plane for $1\leq i\leq c$.
The link $L^{(n_1,\cdots ,n_c)}$ represented by $D^{(n_1,\cdots ,n_c)}$ is called  
{\it the $(n_1,\cdots ,n_c)$-parallel of $L$}. 
Where $L$ is a knot, that is $c=1$, we call $(n)$-parallel $L^{(n)}$ simply a $n$-parallel, and denote it by $L^n$. 
We call a $2$-parallel of a knot {\it untwisted} 
if the linking number of the $2$ components of the parallel is $0$.
\end{definition}

Examples of $(n_1,\cdots ,n_c)$-parallels of links 
are shown in Figure \ref{hopf} and Figure \ref{trefoil}.

\begin{figure}[H]
\begin{center}
\includegraphics[width=8cm]{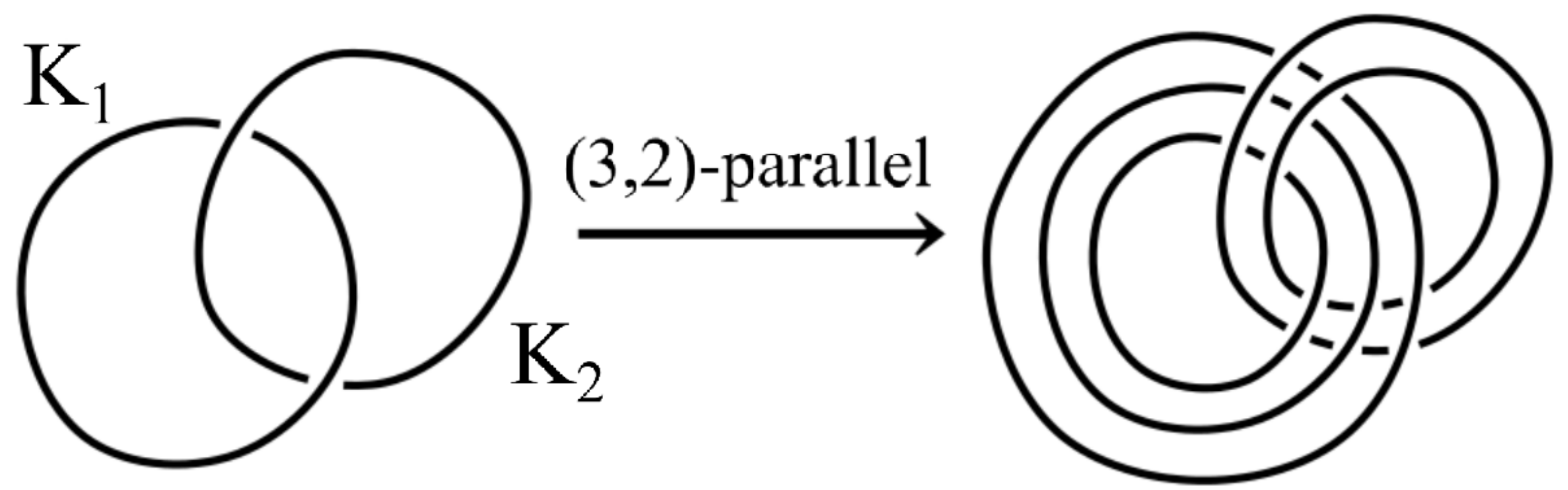}
\caption{A $(3,2)$-parallel of the Hopf link}\label{hopf}
\end{center}
\end{figure}
\begin{figure}[H]
\begin{center}
\includegraphics[width=7cm]{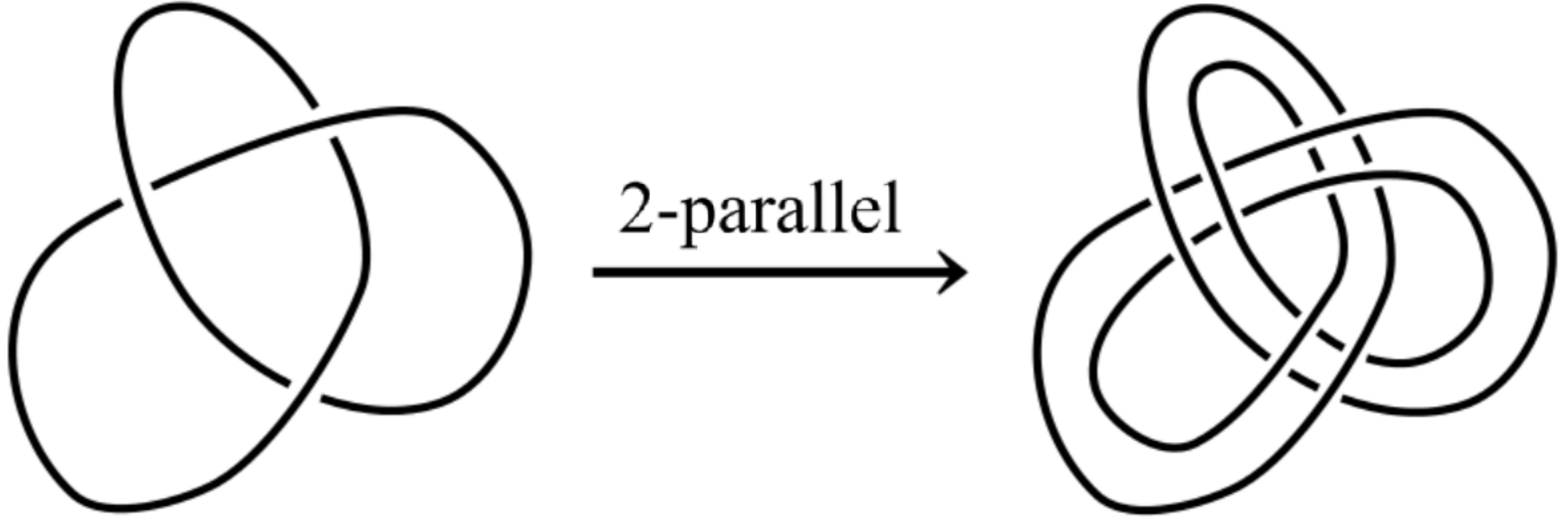}
\caption{A $2$-parallel of the trefoil}\label{trefoil}
\end{center}
\end{figure}

We show that an even parallel of a link is $\mathbb{Z}$-colorable except for the case of 2 parallels with non-zero linking number. 

\begin{theorem}\label{thmparallel}
[1] For a non-trivial knot $K$ 
and any diagram $D$ of $K$ that the writhe is $0$, 
$D^{2}$ always represents a $\mathbb{Z}$-colorable link. 
Moreover, there exists a diagram $D_0$ of $K$ 
such that $D_0^{2}$ is locally equivalent 
to a minimally $\mathbb{Z}$-colorable diagram. 

\noindent
[2] Let $L$ be a non-splittable $c$-component link  
and $D$ any diagram of $L$. 
For any even number $n_1,\cdots,n_c$ at least $4$, 
$D^{(n_1,\cdots ,n_c)}$ always represents a $\mathbb{Z}$-colorable link 
and is locally equivalent to a minimally $\mathbb{Z}$-colorable diagram.
\end{theorem}

Here we give the definitions used in Theorem \ref{thmparallel}.

\begin{definition}
Let $L$ be a $\mathbb{Z}$-colorable link, and $D$ a diagram of $L$. 
$D$ is called a {\it minimally $\mathbb{Z}$-colorable diagram} 
if there exists a $\mathbb{Z}$-coloring $\gamma$ on $D$ 
such that the image of $\gamma$ is equal to the minimal coloring number of $L$. 
\end{definition}

\begin{definition}
For diagrams $D$ and $D'$ of $L$, 
$D$ is {\it locally equivalent} to $D'$ 
if there exist mutually disjoint open subsets on $\mathbb{R}^2$ 
$U_1, U_2, \cdots ,U_n$ such that 
$D'$ is obtained from $D$ by Reidemeister moves only in $\bigcup_{i=1}^m U_i$. 
\end{definition}

%

To prove Theorem \ref{thmparallel} [1], 
we prepare the next lemma about the linking number of components of $2$-parallel of a knot.

\begin{lemma}\label{lem1}
Let $D$ a diagram of a knot $K$. 
For a 2-parallel $K^2=K_1\cup K_2$ represented by $D^2$, 
the linking number of $K_1$ and $K_2$ is equal to the writhe of $D$.
\end{lemma}

\begin{proof}
Any crossing $c$ on $D$ is replaced by four crossings by taking $2$-parallel copies. 
In the four crossings, 
the two crossings consist of the arcs of only $K_1$ or $K_2$, 
and the other two crossings both $K_1$ and $K_2$. 
\begin{figure}[H]
\begin{center}
\includegraphics[width=7cm]{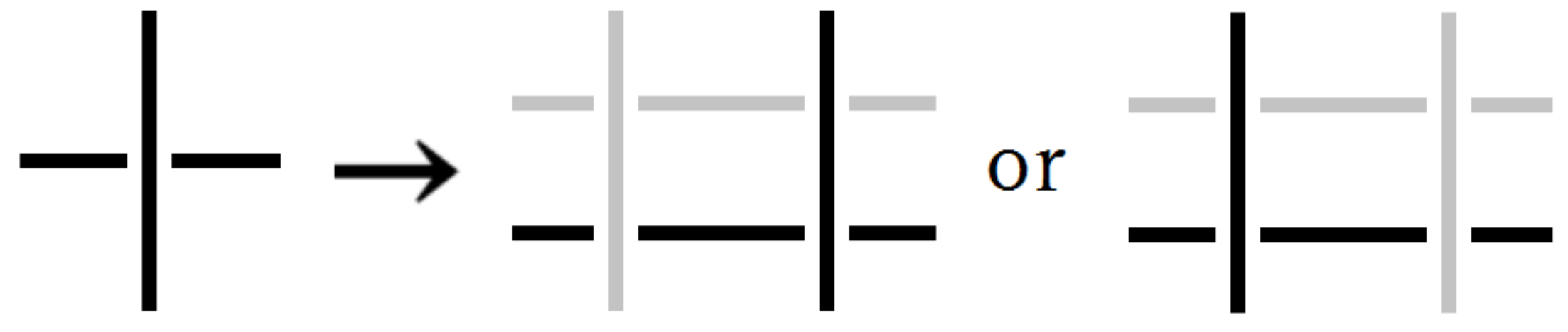}
\caption{}\label{lem1fig}
\end{center}
\end{figure}
Then $c$ and the crossings constructed by the arcs of different components have same sign. 
See Figure \ref{lem1fig}.
Therefore the linking number of $K_1$ and $K_2$ is equal to the writhe of $D$.
\end{proof}

For non-splittability of parallels of knots and links, 
we can also show the next.

\begin{lemma}\label{lem11}
[1] Any $n$-parallel of a non-trivial knot is non-splittable.
[2] Any $(n_1,\cdots ,n_c)$-parallel of a non-splittable link is non-splittable.
\end{lemma}

\begin{proof}

[1] Let $K^n$ be an $n$-parallel of a non-trivial knot $K$. 
Suppose that $K^n$ is splittable, 
that is, there exists a $2$-sphere $S$ in $S^3-K^n$ 
such that $S$ do not bound any $3$-ball in $S^3-K^n$. 
On the other hand, 
there exists an embedded annulus $A$ in $S^3$ such that $K^n \subset A$ 
and the core of $A$ is parallel to the components of $K^n$. 

By isotopy of $S$, let $S\cap A$ be minimized.
If $S \cap A$ is the empty set 
then it is in contradiction with the definition of $S$, 
for any knot complement in $S^3$ is irreducible. 

We suppose otherwise, that is, $S\cap A$ is not the empty set. 
We consider a component $C$ of $S\cap A$ on $A$.
Then there are two possibilities;
(1) $C$ is trivial on $A$, that is , $C$ bounds a disk on $A$, or\\
(2) $C$ is parallel to a component of $K^n$.

In the case (1), 
$C$ can be removed by isotopy of $S$. 
That is in contradiction with that $S\cap A$ minimized.

In the case (2), 
since $C$ bounds a disk in $S$, 
$K$ must be trivial. 
This also gives a contradiction.

Therefore we see $K^n$ is non-splittable.

\bigskip

\noindent
[2] From a link $L=K_1 \cup\cdots\cup K_c$ with $c$ at least 2, 
we obtain a parallel $L^{(n_1,\cdots ,n_c)}=K^{n_1}_1 \cup\cdots\cup K^{n_c}_c$. 

We assume $L^{(n_1,\cdots ,n_c)}$ is splittable. 
Then there exists a $2$-sphere in $S^3-L^{(n_1,\cdots ,n_c)}$  
such that $S^3=B_1\cup B_2$ and $L^{(n_1,\cdots ,n_c)}=L_1 \cup L_2$ with a link $L_i \subset B_i$ for $i=1,2$. 
Here we take a component $l_1$ from $L_1$ and $l_2$ from $L_2$.

In the case that $l_1 \subset K^{n_i}_i$ and $l_2 \subset K^{n_j}_j$ with $i\neq j$. 
Take $l_k \subset K^{n_k}_k$ for each $k\neq i,j$. 
Then we see $L'=(l_1\cup \l_2)\cup \bigcup_{k\neq i,j}l_k$ 
is equivalent to $L^{(n_1,\cdots ,n_c)}$. 
Now $S$ splits $l_1$ and $l_2$. 
From $L'$ is ambient isotopy to $L^{(n_1,\cdots ,n_c)}$, that is contradictory with the assumption.

In the case that $l_1, l_2 \subset K^{n_i}_i$. 
Take $l_k \subset K^{n_k}_k$ with $k\neq i$. 
By $l_1 \subset B_1$, 
together with the assumption that $L$ is non-splittable, 
the link $l_1\cup \bigcup_{k\neq i} l_k$ is contained in $B_1$. 
On the other hand, 
for $l_2$, 
we have $l_2 \cup \bigcup_{k\neq i}l_k$ is contained in $B_2$. 
That is contradictory to each other. 
Therefore $L^{(n_1,\cdots ,n_c)}$ is non-splittable.
\end{proof}

\begin{proof}[Proof of Theorem \ref{thmparallel}]
[1] Let $K$ be a non-trivial knot in $S^3$. 
First we give an orientation to the knot $K$. 
Let us take a diagram $D$ of $K$. 
Consider the $2$-parallel $K^2$ obtained from $D$. 
Then the $2$-parallel $K^2=K_1\cup K_2$ admits the orientation 
induced from that of $K$. 
Suppose that $K^2$ is untwisted, 
i.e., the linking number of $K_1$ and $K_2$ is $0$.
Then we see the writhe of $D$ is $0$ by Lemma \ref{lem1}. 
Note that $D^2$ has parallel arcs such that the crossings in both ends have positive signs,  
there exists the same number of parallel arcs such that the crossings in both ends have negative signs 
as shown in Figure \ref{same}.

\begin{figure}[H]
\begin{center}
\includegraphics[width=9cm,clip]{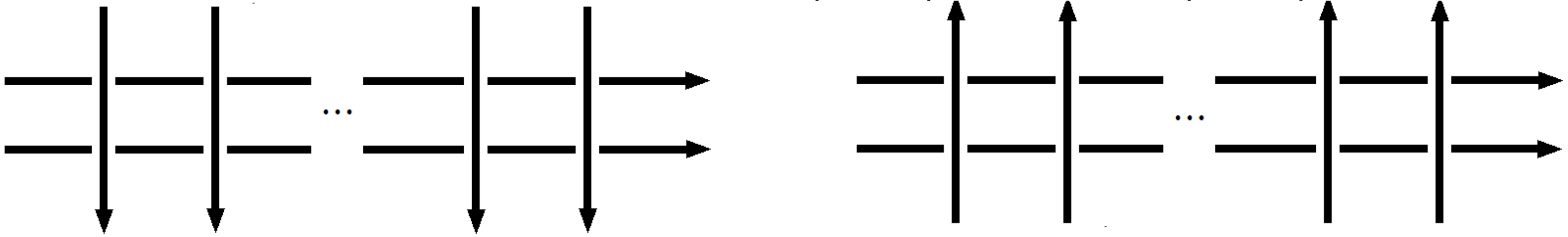}
\caption{}\label{same}
\end{center}
\end{figure}

Here we add a full-twist to the parallel arcs with the sign 
as shown in Figure \ref{r1a} and Figure \ref{r1b}. 

\begin{figure}[H]
\begin{center}
\includegraphics[width=10cm,clip]{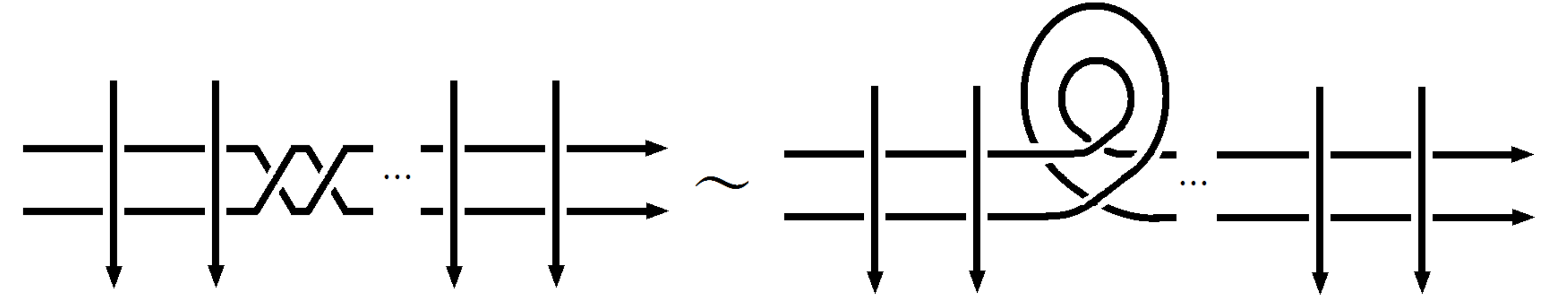}
\caption{}\label{r1a}
\includegraphics[width=10cm,clip]{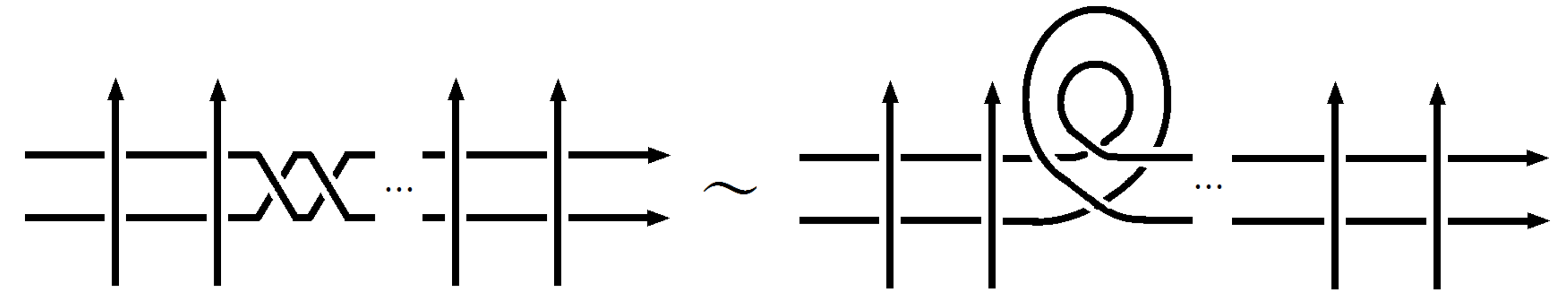}
\caption{}\label{r1b}
\end{center}
\end{figure}another

Since the writhe of $D$ is $0$, 
there exist positive crossings 
as many as negative crossings on $D$. 
Therefore the diagram obtained by this modification 
is equivalent to $D^2$. 
We see that $D^2$ represents $K^2$.
In the following, 
we will use the same notation $D^2$ to denote the modified diagram for convenience. 

We set the colors (integers) $a$ and $a+d$ to a pair of parallel arcs on $D^2$. 
That is $d$ is a difference of colors of two parallel arcs. 
We give the colors to remaining arcs around the arcs to satisfy the condition of $\mathbb{Z}$-coloring.  
See Figure \ref{diff}. 
Then $d$ stays constant after passing under another two parallel arcs. 

\begin{figure}[H]
\begin{center}
\includegraphics[width=4cm,clip]{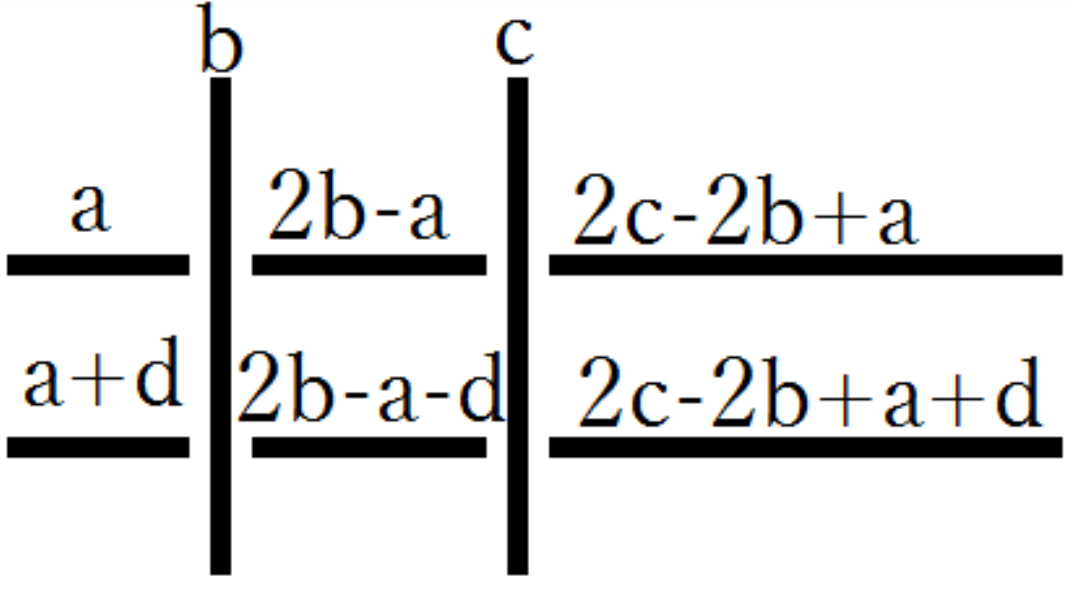}
\caption{}\label{diff}
\end{center}
\end{figure}

Therefore, $D$ has only crossings shown in Figures \ref{2+} and \ref{2-}. 
There $x_i$ and $y_i$ with $i=1,2$ are the arcs of the same component.

\begin{figure}[H]
\begin{center}
\includegraphics[height=3cm,clip]{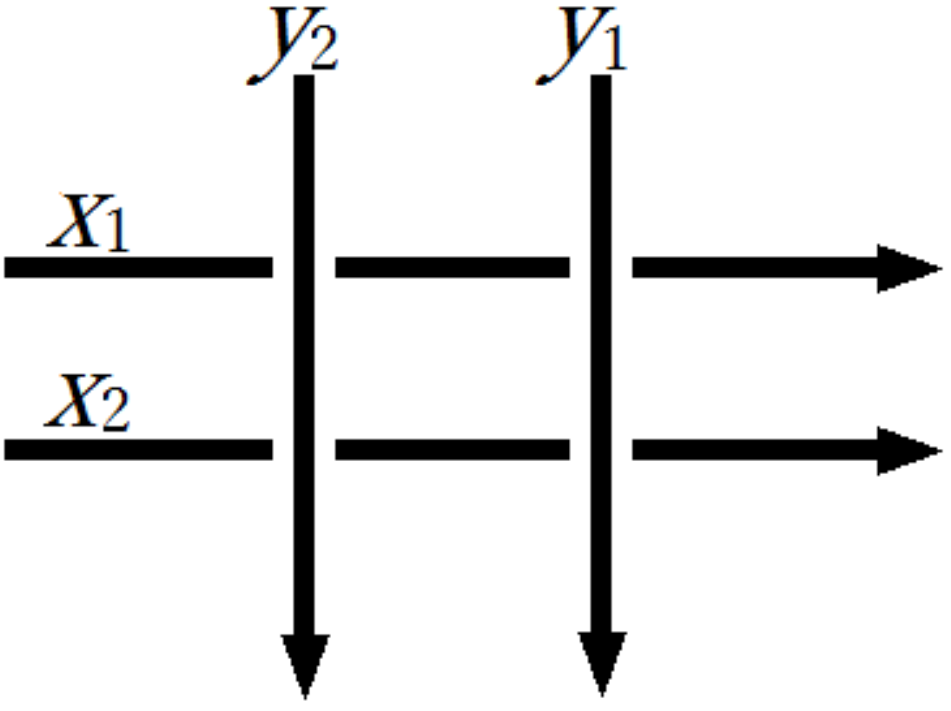}
\caption{}\label{2+}
\includegraphics[height=3cm,clip]{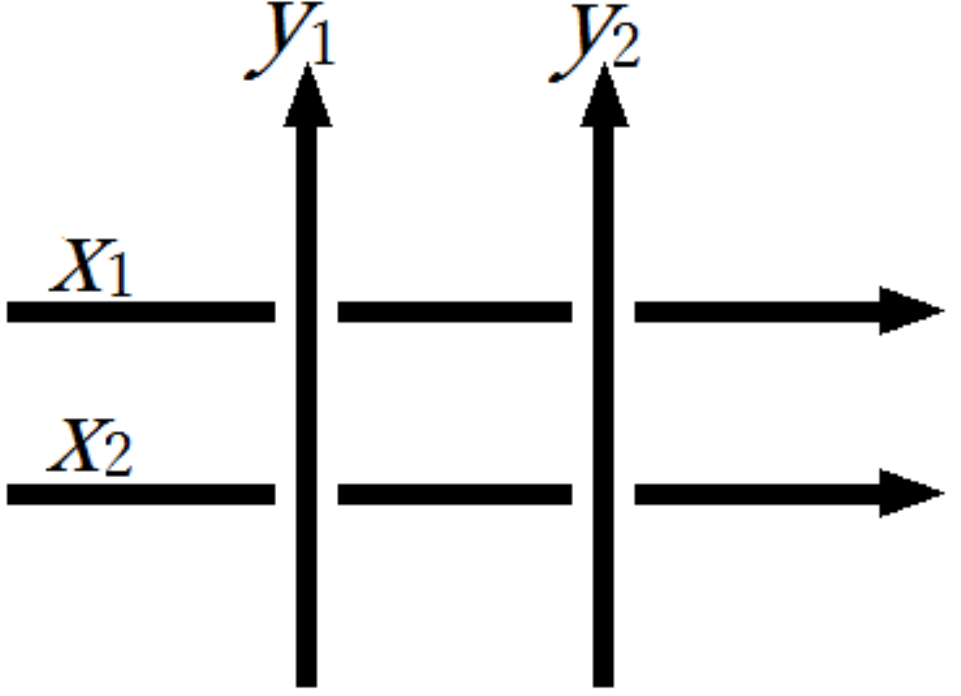}
\caption{}\label{2-}
\end{center}
\end{figure} 

They appear alternately with tracing two parallel arcs. 
We fix the colors of two parallel arcs are $0$ and $1$.
For the arc colored by $0$ is changed to be colored by $1$ 
after passing under another two parallel arcs. 
Then we see the colors as shown in Figure \ref{diff2}. 

\begin{figure}[H]
\begin{center}
\includegraphics[height=3cm,clip]{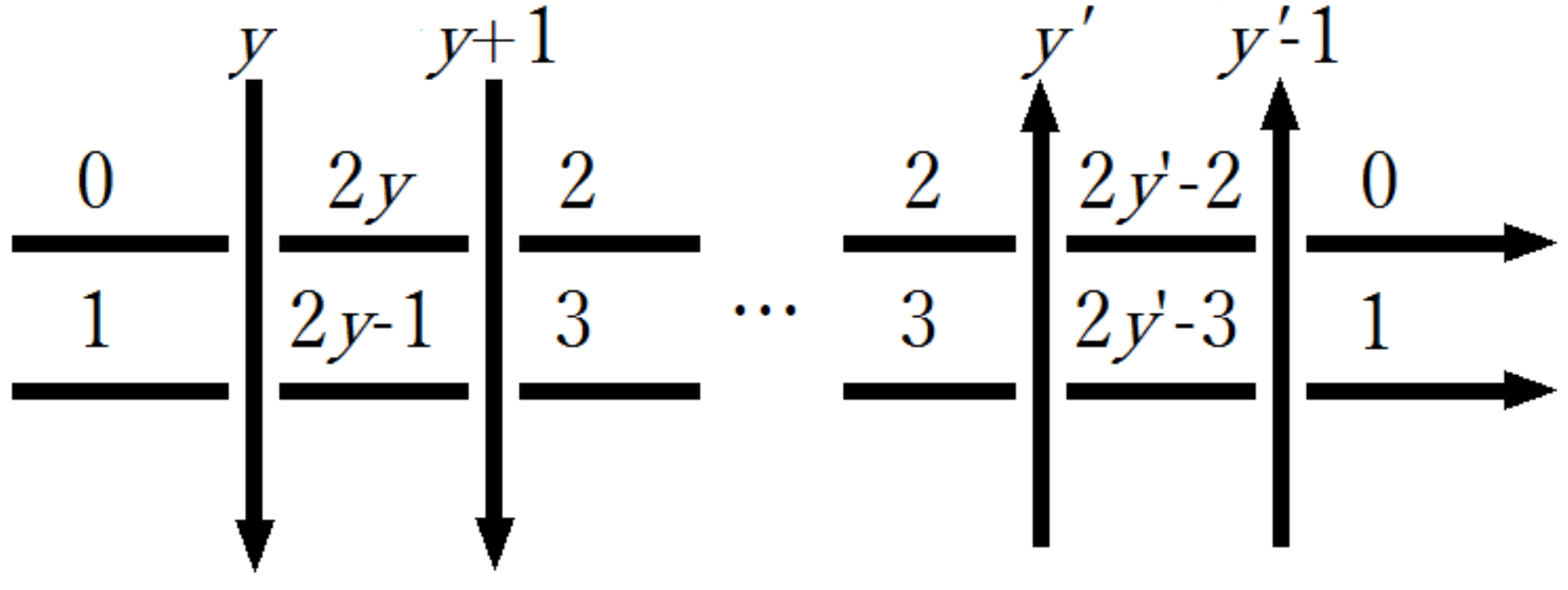}
\caption{}\label{diff2}
\end{center}
\end{figure}
We obtain $y=0, 2$ and $y'=1, 3$. 
We see $2y=0$ or $4$, $2y-1=-1$ or $3$, $2y'-2=0$ or $4$ and $2y'-3=-1$ or $3$. 
We see that 
$D^2$ admits a $\mathbb{Z}$-coloring $C$ such that Im$(C)=\{-1,0,1,2,3,4\}$. 
Therefore $K^2$ is $\mathbb{Z}$-colorable.\\

Here we focus on the arc colored by $4$ or $-1$. 
We can get the $\mathbb{Z}$-coloring without the colors $4$ and $-1$ 
by changing the diagram and the coloring obtained above 
as shown in Figure \ref{del4} and Figure \ref{del-1}.
\begin{figure}[H]
\begin{center}
\includegraphics[height=2cm,clip]{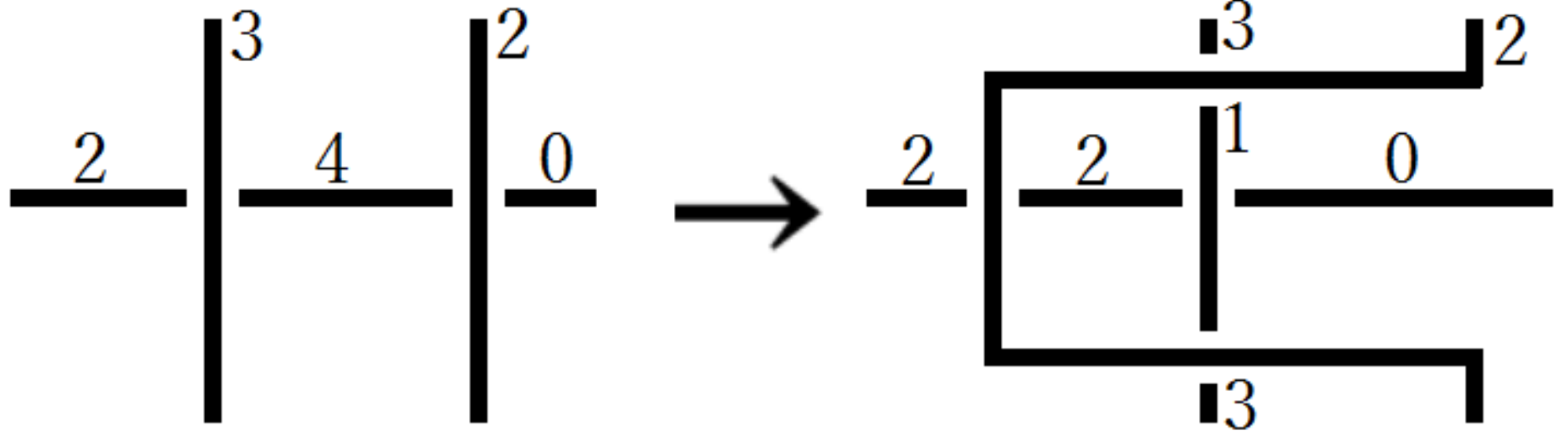}
\caption{}\label{del4}
\end{center}
\end{figure}
\begin{figure}[H]
\begin{center}
\includegraphics[height=2cm,clip]{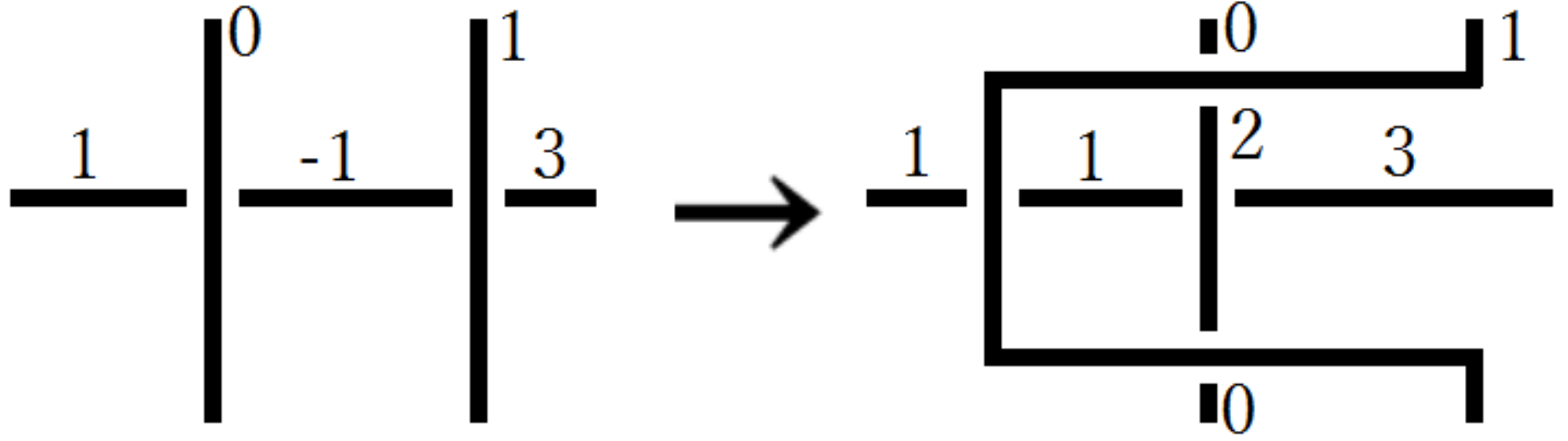}
\caption{}\label{del-1}
\end{center}
\end{figure}
We see that $K^2$ also admits the coloring $C'$ such that Im$(C')=\{0,1,2,3\}$. 
%
By Lemma \ref{lem11}, $K^2$ is non-splittable. 
By Theorem \ref{main}, the obtained diagram is a minimally $\mathbb{Z}$-colorable diagram, 
which is locally equivalent to $D^2$.


%
%
%
%

\bigskip

\noindent
[2] Let $D^{(n_1,\cdots ,n_c)}$ be a diagram of $L^{(n_1,\cdots ,n_c)}$. 
The diagram $D^{(n_1,\cdots ,n_c)}$ has only crossings 
shown in Figure \ref{N-parallel}.

\begin{figure}[H]
\begin{center}
\includegraphics[height=4cm,clip]{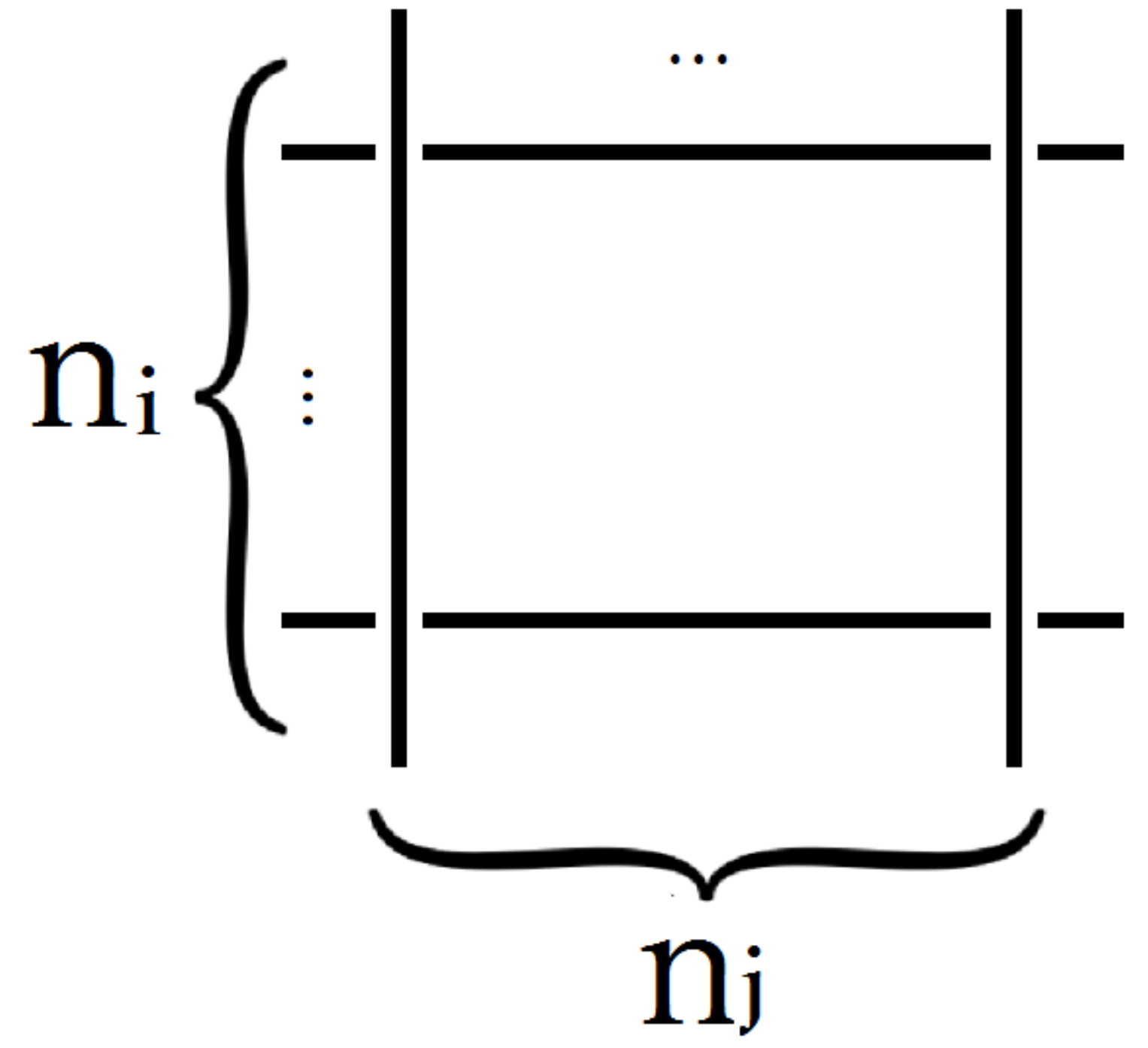}
\caption{}\label{N-parallel}
\end{center}
\end{figure}

We put a circle as fencing the crossings as shown in Figure \ref{circle}.

\begin{figure}[H]
\begin{center}
\includegraphics[height=5cm,clip]{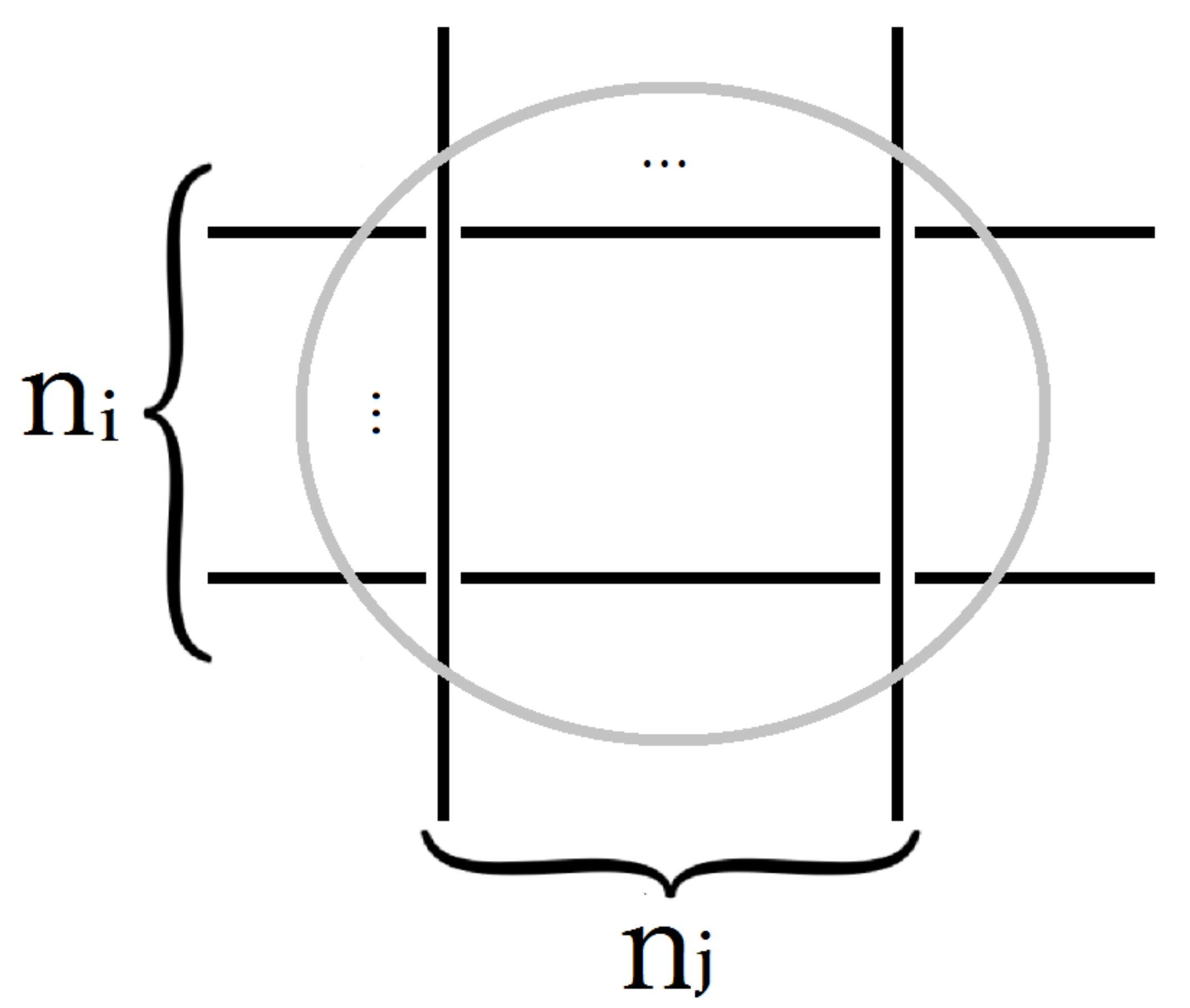}
\caption{}\label{circle}
\end{center}
\end{figure}

Note that each crossing of $D^{(n_1,\cdots,n_k)}$ is contained 
in some of the regions encircled.

For any parallel family of arcs $\{a_1,\cdots,a_k\}$ outside the regions, 
we fix the colors of  $a_{k/2}$ and $a_{k/2+1}$ are $1$ and others are $0$ as shown in Figure \ref{out}. \\

\begin{figure}[H]
\begin{center}
\includegraphics[height=4cm,clip]{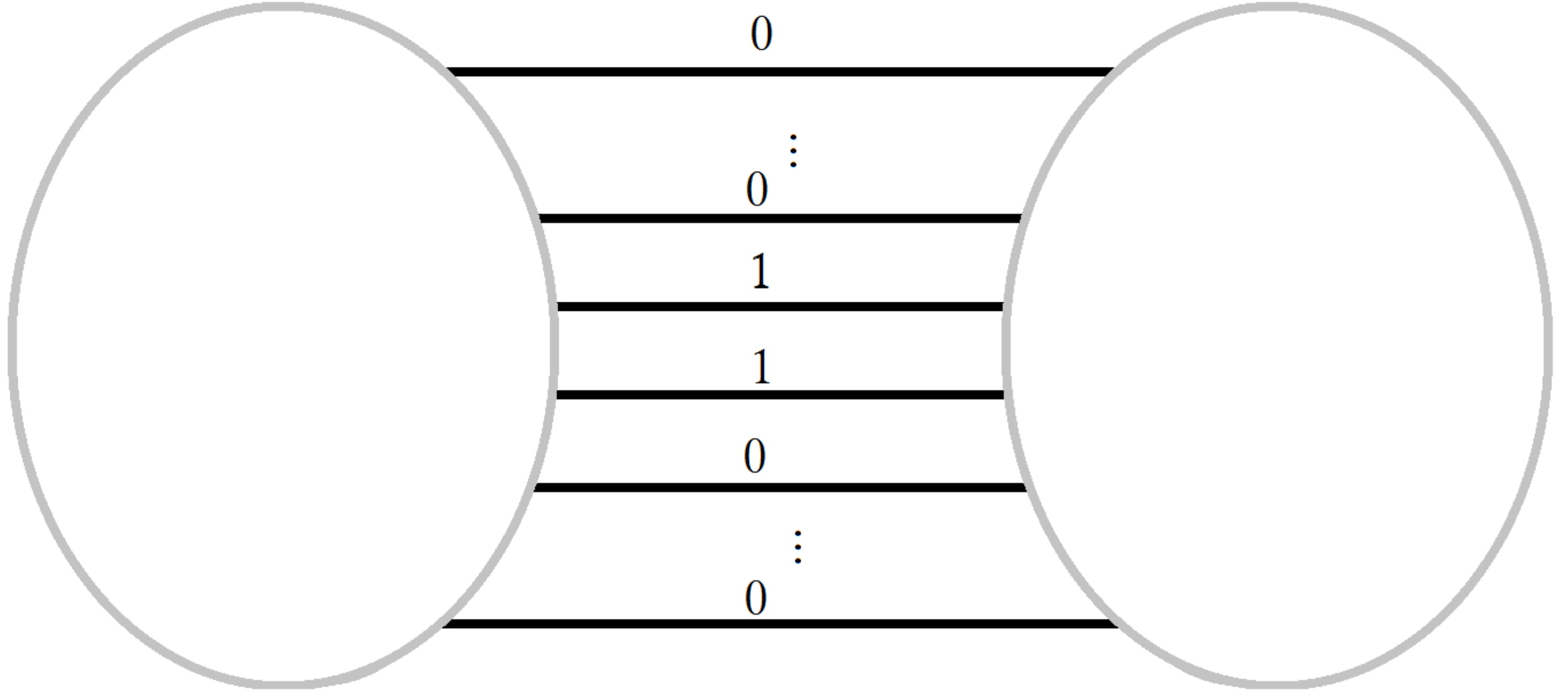}
\caption{}\label{out}
\end{center}
\end{figure}

For the arcs inside the region, we assign colors as follows. \\

In the case $n_j=4m$ for some integer $m$,
we assign the colors $-1,0,1,2,3$ as shown in Figure \ref{4m}.

\begin{figure}[H]
\begin{center}
\includegraphics[height=7cm,clip]{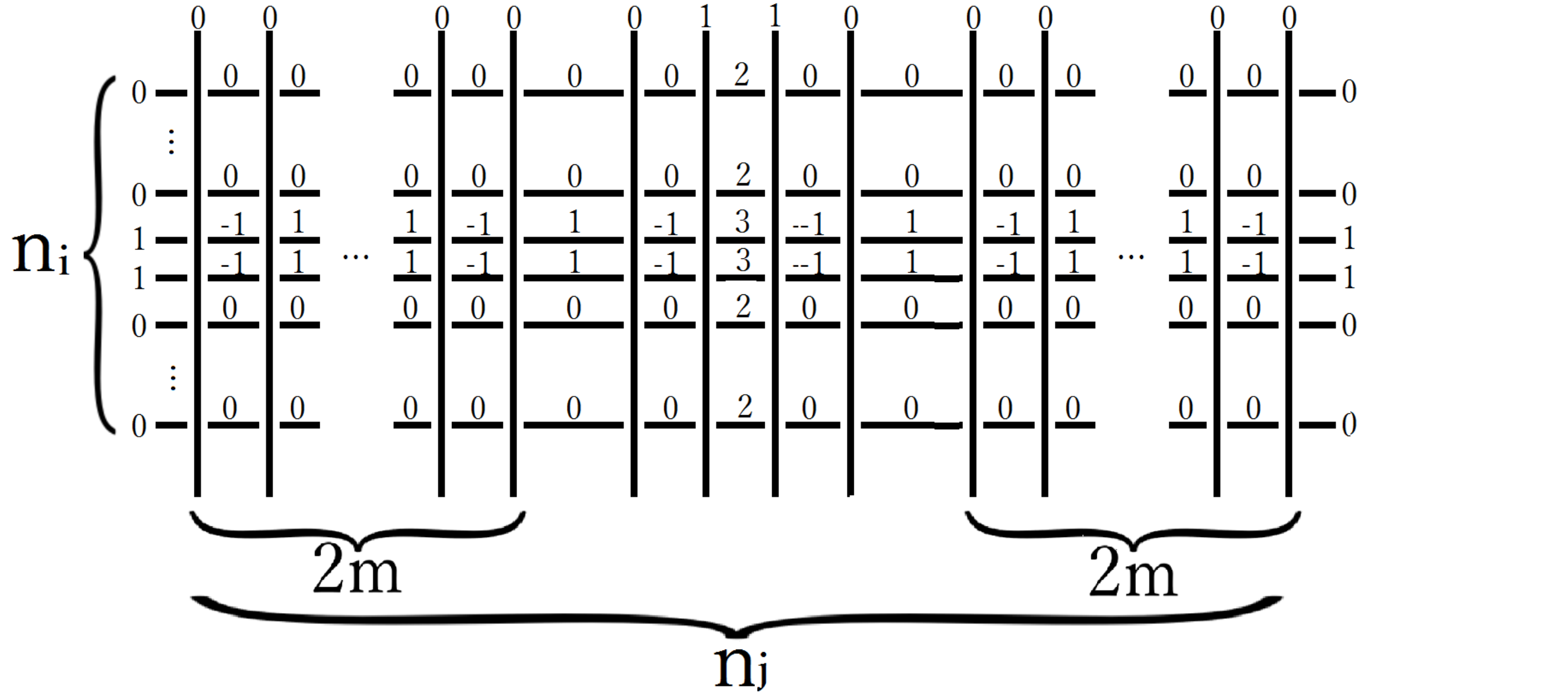}
\caption{}\label{4m}
\end{center}
\end{figure}

Then, at each crossing inside the region, 
the obtained coloring satisfy the condition of $\mathbb{Z}$-coloring. 
Here the colors of the arcs intersecting the circle 
are compatible to those of the arcs outside. \\

In the case $n_j=4m+2$, in the same way, 
we assign the colors $-1,0,1,2$ as shown in Figure \ref{4m+2}.

\begin{figure}[H]
\begin{center}
\includegraphics[height=7cm,clip]{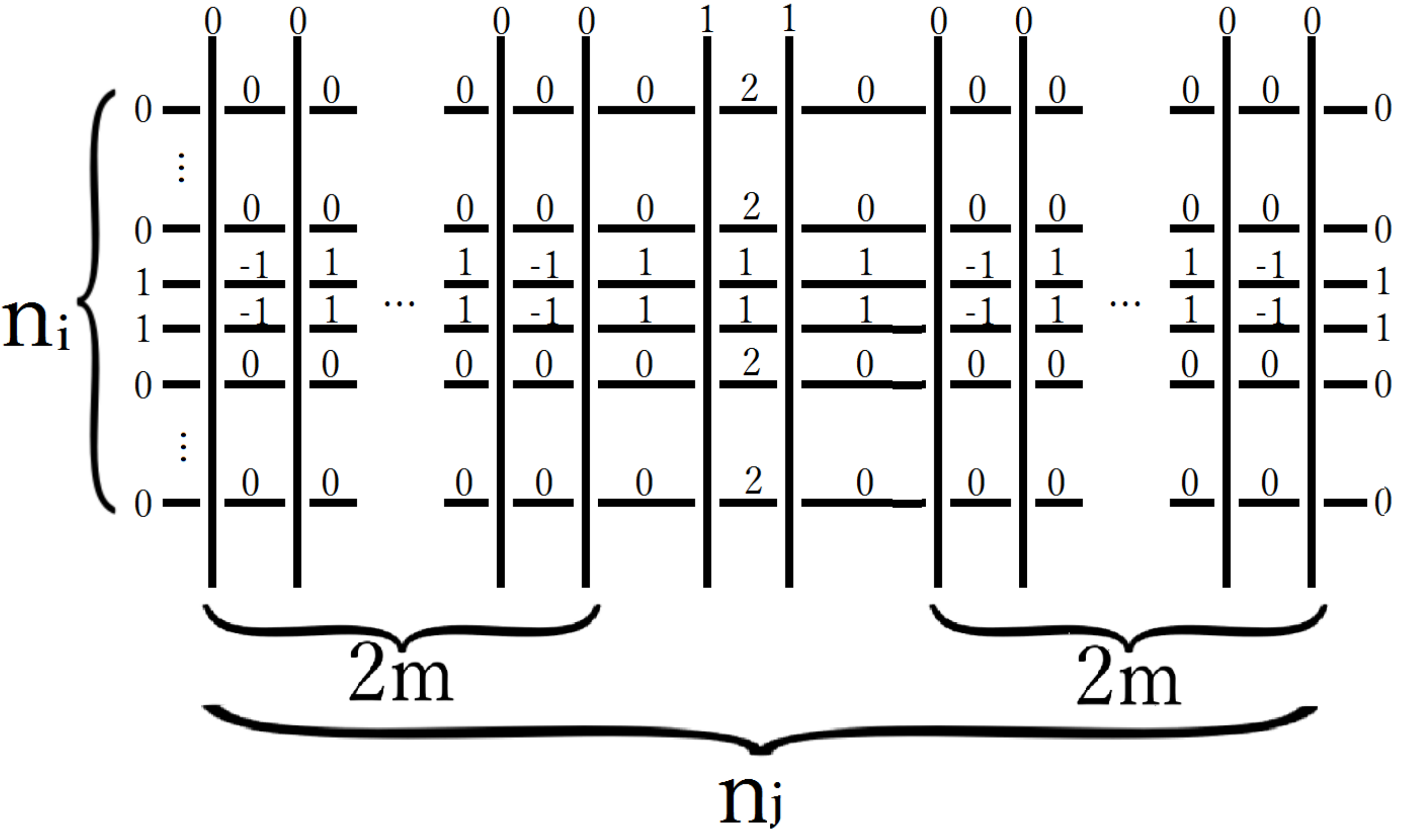}
\caption{}\label{4m+2}
\end{center}
\end{figure}

Then, at each crossing inside the region, 
the obtained coloring satisfy the condition of $\mathbb{Z}$-coloring. 
Here the colors of the arcs intersecting the circle are 
compatible to those of the arcs outside. 

We see that 
$D^{(n_1,\cdots ,n_c)}$ admits a $\mathbb{Z}$-coloring $C$ 
such that Im$(C)=\{-1,0,1,2\}$ or $\{-1,0,1,2,3\}$. 
Therefore $L^{(n_1,\cdots ,n_c)}$ is $\mathbb{Z}$-colorable.\\

Moreover, in the case that the colors $3$ appears,
we delete the color $3$ by changing the diagram and the coloring 
as shown in Figure \ref{2-del3}.

\begin{figure}[H]
\begin{center}
\includegraphics[height=2cm,clip]{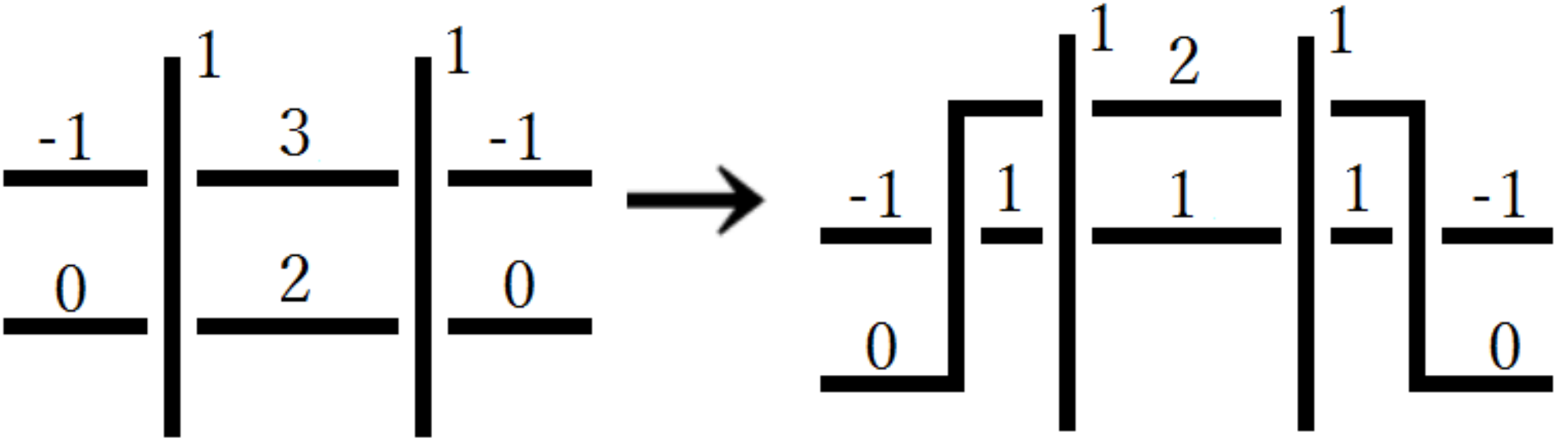}
\caption{Delete the color $3$}\label{2-del3}
\end{center}
\end{figure}

Therefore there exists the coloring $C'$ such that Im$(C')=\{-1,0,1,2\}$. 
%
%

Since we are assuming that L is non-splittable and Lemma \ref{lem11}, 
it can be shown that the parallel is non-splittable. 
By Theorem \ref{main}, the obtained diagram is a minimally $\mathbb{Z}$-colorable diagram, 
which is locally equivalent to $D^{(n_1,\cdots ,n_c)}$.

\end{proof}

\section{Acknowledgement}
I would like to thank Meiqiao Zhang for useful discussions. 
I would also like to thank Akio Kawauchi 
for giving me the motivation that I consider about parallels. 
I am grateful to Kouki Taniyama 
for pointing out the necessity of Lemmas \ref{lem11}. 
I am also grateful to Ayumu Inoue, Takuji Nakamura and Shin Satoh 
for useful discussions and advices. 
I would also like to Kazuhiro Ichihara for his supports and advices.




\begin{thebibliography}{0}

\bibitem{Fox}
R. H. Fox, 
A quick trip through knot theory, 
in {\it Topology of 3-manifolds and related topics 
(Proc. The Univ. of Georgia Institute, 1961)}, 
120--167, Prentice Hall, Englewood Cliffs, NJ.

\bibitem{HK}
F. Harary\ and\ L. H. Kauffman, 
Knots and graphs. I. Arc graphs and colorings, 
Adv. in Appl. Math. 22 (1999), no.~3, 312--337. 

\bibitem{IM}
K. Ichihara\ and\ E. Matsudo, 
Minimal coloring number for Z-colorable links, 
J. Knot Theory Ramifications {\bf 26} (2017), no.~4, 1750018, 23 pp.

\bibitem{Zhang}
M. Zhang, X. Jin, and Q. Deng, 
The Minimal Coloring Number Of Any Non-splittable Z-colorable Link Is Four, 
preprint, arXiv:1706.08837.



\end{thebibliography}
\end{document}